\documentclass[11pt]{amsart}
\usepackage[margin=1in]{geometry}

\usepackage{amssymb}
\usepackage{amsthm}
\usepackage{amsmath}
\usepackage{mathrsfs}
\usepackage{amsbsy}
\usepackage[all]{xy}
\usepackage{bm}
\usepackage{hyperref}
\usepackage{tikz}
\usepackage{array}
\usepackage{float}
\usepackage{enumerate}
\usepackage{xcolor}
\usepackage{hhline}
\allowdisplaybreaks
\usepackage[noadjust]{cite}

\usepackage{caption}
\usepackage{tabu}
\usepackage{diagbox}

\usepackage[noabbrev,capitalise,nameinlink]{cleveref}

\hypersetup{colorlinks=true, citecolor=lavender, linkcolor=lavender,urlcolor=lavender}

\definecolor{lavender}{rgb}{0.5,0,1.0}

\newenvironment{enumerate*}
  {\begin{enumerate}[(I)]
    \setlength{\itemsep}{10pt}
    \setlength{\parskip}{0pt}}
  {\end{enumerate}}

\newtheorem{theorem}{Theorem}[section]
\newtheorem{proposition}[theorem]{Proposition}

\newtheorem{lemma}[theorem]{Lemma}

\theoremstyle{definition}
\newtheorem{definition}[theorem]{Definition}
\newtheorem{example}[theorem]{Example}
\newtheorem{remark}[theorem]{Remark}

\newcommand{\dfn}[1]{\textcolor{blue}{\emph{#1}}}

\newcommand{\SortNoop}[1]{}

\usepackage{doi}

\newcommand{\ddeg}{\mathrm{ddeg}}
\newcommand{\dTV}{d_{\mathrm{TV}}}
\newcommand{\tmix}{t^{\mathrm{mix}}}
\newcommand{\width}{\mathrm{width}}

\newcommand{\mc}{\mathcal}
\newcommand{\row}{\mathrm{Row}}
\newcommand{\TT}{\mathbf{T}}
\newcommand{\JJ}{\mathcal{J}}

\newcommand{\MM}{\mathcal{M}}
\newcommand{\HH}{\mathcal{H}}
\newcommand{\KK}{\mathcal{K}}
\newcommand{\bM}{{\bf M}}

\newcommand{\Ind}{\mathrm{Ind}}
\newcommand{\DD}{\mathcal{D}}
\newcommand{\UU}{\mathcal{U}}

\usepackage{etoolbox}

\newcommand{\includeSymbol}[1]{\ensuremath{%
	\mathchoice
		{\raisebox{-.7mm}{\includegraphics[height=2.2ex]{#1}}}	
		{\raisebox{-.7mm}{\includegraphics[height=2.2ex]{#1}}}
		{\raisebox{-.6mm}{\includegraphics[height=1.6ex]{#1}}}
		{\raisebox{-.5mm}{\includegraphics[height=1ex]{#1}}}
}}

\robustify{\includeSymbol}
\newcommand{\hexx}{\includeSymbol{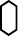}}

\def\P{\mathbb{P}}

\newcommand{\expect}[1]{\mathbb{E}\left(#1\right)}

\newcommand{\varn}{\mathrm{Var}} 

\definecolor{NormalGreen}{RGB}{0,220,0}
\definecolor{MyPurple}{RGB}{200,0,255}

\usepackage{todonotes}

\begin{document}

\title{Rowmotion Markov Chains}
\subjclass[2010]{}

\author[Colin Defant]{Colin Defant}
\address[]{Department of Mathematics, Massachusetts Institute of Technology, Cambridge, MA 02139, USA}
\email{colindefant@gmail.com}

\author[Rupert Li]{Rupert Li}
\address[]{Massachusetts Institute of Technology, Cambridge, MA 02139, USA}
\email{rupertli@mit.edu}

\author[Evita Nestoridi]{Evita Nestoridi}
\address[]{Mathematics Department, Stony Brook University, Stony Brook NY, 11794-3651, USA}
\email{evrydiki.nestoridi@stonybrook.edu}

\maketitle

\begin{abstract}
\emph{Rowmotion} is a certain well-studied bijective operator on the distributive lattice $J(P)$ of order ideals of a finite poset $P$. We introduce the \emph{rowmotion Markov chain} ${\bf M}_{J(P)}$ by assigning a probability $p_x$ to each $x\in P$ and using these probabilities to insert randomness into the original definition of rowmotion. More generally, we introduce a very broad family of \emph{toggle Markov chains} inspired by Striker's notion of generalized toggling. We characterize when toggle Markov chains are irreducible, and we show that each toggle Markov chain has a remarkably simple stationary distribution. 

We also provide a second generalization of rowmotion Markov chains to the context of semidistrim lattices. Given a semidistrim lattice $L$, we assign a probability $p_j$ to each join-irreducible element $j$ of $L$ and use these probabilities to construct a rowmotion Markov chain ${\bf M}_L$. Under the assumption that each probability $p_j$ is strictly between $0$ and $1$, we prove that ${\bf M}_{L}$ is irreducible. We also compute the stationary distribution of the rowmotion Markov chain of a lattice obtained by adding a minimal element and a maximal element to a disjoint union of two chains.

We bound the mixing time of ${\bf M}_{L}$ for an arbitrary semidistrim lattice $L$. In the special case when $L$ is a Boolean lattice, we use spectral methods to obtain much stronger estimates on the mixing time, showing that rowmotion Markov chains of Boolean lattices exhibit the cutoff phenomenon.   
\end{abstract}

\section{Introduction}\label{sec:intro}

\subsection{Distributive Lattices}\label{subsec:distributive}
Let $P$ be a finite poset, and let $J(P)$ denote the set of order ideals (i.e., down-sets) of $P$. For $S\subseteq P$, let \[\Delta(S)=\{x\in P:x\leq s\text{ for some }s\in S\}\quad\text{and}\quad\nabla(S)=\{x\in P:x\geq s\text{ for some }s\in S\},\] and let $\min(S)$ and $\max(S)$ denote the set of minimal elements and the set of maximal elements of $S$, respectively. \dfn{Rowmotion}, a well-studied operator in the growing field of dynamical algebraic combinatorics, is the bijection $\row\colon J(P)\to J(P)$ defined by\footnote{Many authors define rowmotion to be the inverse of the operator that we have defined. Our definition agrees with the conventions used in \cite{Barnard, BarnardHanson, Semidistrim, ThomasWilliams}.} 
\begin{equation}\label{eq:row_def}
\row(I)=P\setminus \nabla(\max(I)).
\end{equation} 
We refer the reader to \cite{StrikerWilliams,ThomasWilliams} for the history of rowmotion. The purpose of this article is to introduce randomness into the ongoing saga of rowmotion by defining certain Markov chains. We were inspired by the articles \cite{Ayyer, Poznanovic, Rhodes}; these articles define Markov chains based on the \emph{promotion} operator, which is closely related to rowmotion in special cases \cite{Bernstein,StrikerWilliams} (though our Markov chains are fundamentally different from these promotion-based Markov chains). 

For each $x\in P$, fix a probability $p_x\in[0,1]$. We define the \dfn{rowmotion Markov chain} $\bM_{J(P)}$ with state space $J(P)$ as follows. Starting from a state $I\in J(P)$, select a random subset $S$ of $\max(I)$ by adding each element $x\in\max(I)$ into $S$ with probability $p_x$; then transition to the new state $P\setminus\nabla(S)=\row(\Delta(S))$. Thus, for any $I,I'\in J(P)$, the transition probability from $I$ to $I'$ is \[\mathbb P(I\to I')=\begin{cases} \left(\prod\limits_{x\in \min(P\setminus I')}p_x\right)\left(\prod\limits_{x'\in \max(I)\setminus\min(P\setminus I')}(1-p_{x'})\right) & \mbox{if }\min(P\setminus I')\subseteq\max(I); \\ 0 & \mbox{otherwise.}\end{cases}\]
Observe that if $p_x=1$ for all $x\in P$, then $\bM_{J(P)}$ is deterministic and agrees with the rowmotion operator. On the other hand, if $p_x=0$ for all $x\in P$, then $\bM_{J(P)}$ is deterministic and sends every order ideal of $P$ to the order ideal $P$. 

\begin{example}\label{Exam1}
Suppose $P$ is the poset \[\begin{array}{l}\includegraphics[height=.9cm]{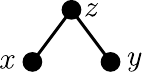}\end{array},\] whose elements $x,y,z$ are as indicated. Then $J(P)$ forms a distributive lattice with $5$ elements. The transition diagram of ${\bf M}_{J(P)}$ is drawn over the Hasse diagram of $J(P)$ in \Cref{Fig1}. 
\end{example}

\begin{figure}[ht]
  \begin{center}{\includegraphics[height=13.391cm]{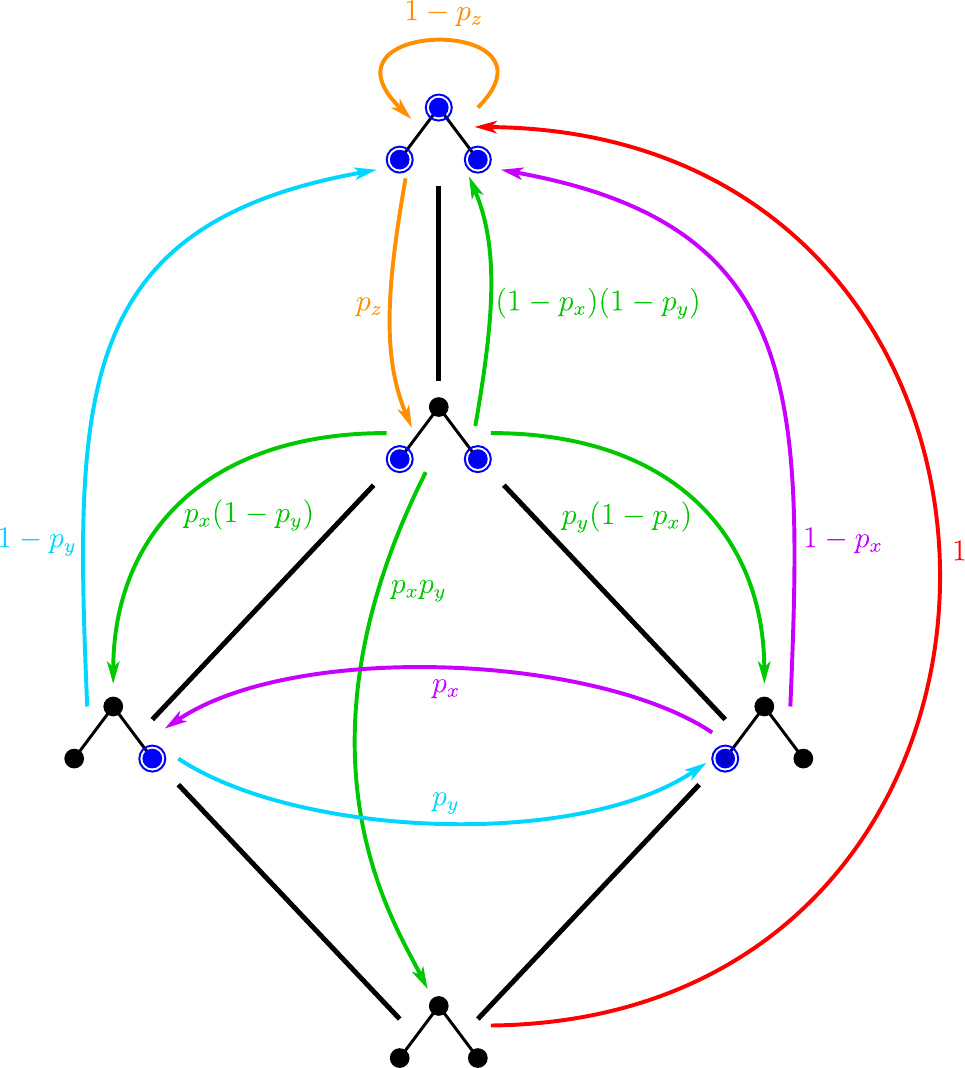}}
  \end{center}
  \caption{The transition diagram of ${\bf M}_{J(P)}$, where $P$ is the $3$-element poset from \Cref{Exam1}. The elements of each order ideal in $J(P)$ are circled and blue. }\label{Fig1}
\end{figure}

Suppose each probability $p_x$ is strictly between 0 and 1. One of our main results will imply that ${\bf M}_{J(P)}$ is irreducible and that the probability of the state $I$ in the stationary distribution of $\bM_{J(P)}$ is \begin{equation}\label{eq:stationary}
\frac{1}{Z(J(P))}\prod_{x\in I}p_x^{-1},
\end{equation} 
where $\displaystyle Z(J(P))=\sum_{I'\in J(P)}\prod_{x'\in I'}p_{x'}^{-1}$.

It is surprising that there is such a clean formula for the stationary distribution in this level of generality. We will deduce this result from a more general result about a vastly broader family of Markov chains. 

\subsection{Toggle Markov Chains}\label{subsec:toggle}

Let $P$ be a finite set of size $n$, and let $\KK$ be a collection of subsets of $P$. For each $x\in P$, define the \dfn{toggle operator} $\tau_x\colon \KK\to \KK$ by \[\tau_x(A)=\begin{cases}
A\triangle\{x\} & \text{ if } A\triangle\{x\}\in \KK \\
A & \text{ otherwise},
\end{cases}\]
where $\triangle$ denotes symmetric difference. Note that $\tau_x$ is an involution. 
Fix a tuple ${\bf x}=(x_1,\ldots,x_n)$ that contains each element of $P$ exactly once. In other words, ${\bf x}$ is an ordering of the elements of $P$. Given a set $Y\subseteq P$, let $\tau_Y=\tau_{y_r}\circ\cdots\circ\tau_{y_1}$, where $y_1,\ldots,y_r$ is the list of elements of $Y$ in the order that they appear within the list $x_1,\ldots,x_n$.

Striker \cite{Striker} viewed the map $\tau_P\colon\KK\to\KK$ as a generalization of rowmotion. Indeed, if $P$ is a poset, ${\bf x}=(x_1,\ldots,x_n)$ is a linear extension of $P$ (meaning $i<j$ whenever $x_i<x_j$ in $P$), and $\KK=J(P)$, then $\tau_P$ is equal to rowmotion. The recent article \cite{Elder} studies the dynamical aspects of $\tau_P$ when $P$ is a poset, ${\bf x}$ is a linear extension of $P$, and $\KK$ is the collection of \emph{interval-closed} (also called \emph{convex}) subsets of $P$. The articles \cite{DJMM, Joseph, JosephRoby} consider $\tau_P$ when $P$ is the vertex set of a particular graph, ${\bf x}$ is a special ordering of the vertices, and $\KK$ is the collection of independent sets of the graph. 

For each $x\in P$, fix a probability $p_x$. Define the \dfn{toggle Markov chain} $\TT=\TT(\KK,{\bf x})$ as follows. The state space of $\TT$ is $\KK$. Suppose the Markov chain is in a state $A\in \KK$. Choose a subset $T\subseteq A$ randomly so that each element $x\in A$ is included in $T$ with probability $p_x$, and then transition from $A$ to the new state $\tau_{T}(A)$. 

To phrase this differently, define the \dfn{random toggle} $\widetilde\tau_x$ to be the stochastic operator that acts as follows on a set $A\in\KK$. Let $X$ be a Bernoulli random variable that takes the value $1$ with probability $p_x$, and let \[\widetilde\tau_x(A)=\begin{cases} \tau_x(A) & \mbox{if }x\not\in A\text{ or }X=1; \\ A & \mbox{if }x\in A\text{ and }X=0.\end{cases}\] Then the Markov chain transitions from the state obtained from $A$ by applying the random toggles $\widetilde\tau_{x_1},\ldots,\widetilde\tau_{x_n}$ in this order. (Each time we apply a random toggle, we use a new Bernoulli random variable that is independent of those used before.) 

\begin{example}\label{Exam2}
Suppose $G$ is the graph \[\begin{array}{l}\includegraphics[height=0.551cm]{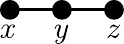}\end{array},\] whose vertices $x,y,z$ are as indicated. Let $\KK$ be the collection of independent sets of $G$. \Cref{Fig4} depicts the random toggles $\widetilde\tau_x,\widetilde\tau_y,\widetilde\tau_z$. If we let ${\bf x}=(x,y,z)$, then a transition of ${\bf T}(\KK,{\bf x})$ consists of applying these random toggles in the order $\widetilde\tau_x,\widetilde\tau_y,\widetilde\tau_z$. 
\end{example}

\begin{figure}[ht]
  \begin{center}{\includegraphics[height=4.818cm]{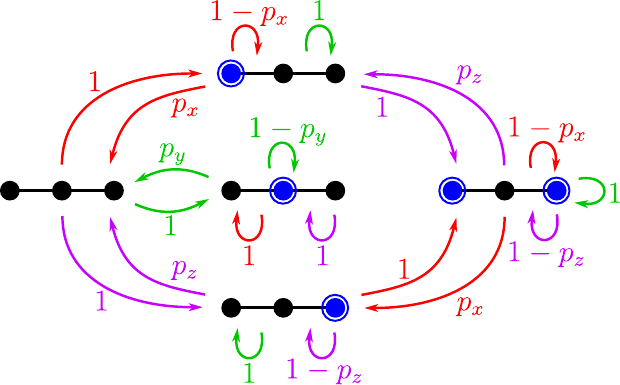}}
  \end{center}
  \caption{As in \cref{Exam2}, we consider random toggles, where $\KK$ is the collection of independent sets of a path graph with vertices $x,y,z$ (from left to right). The elements of each independent set are circled and blue. To apply the random toggle $\widetilde\tau_x$ to an independent set $A$, we follow one of the {\color{red}red} arrows starting at $A$; the probability that a particular arrow is used is written next to the arrow. Similarly, we follow a {\color{NormalGreen}green} arrow when we apply $\widetilde\tau_y$, and we follow a {\color{MyPurple}purple} arrow when we apply $\widetilde\tau_z$. }\label{Fig4}
\end{figure}

Given a set $P$, let $\HH^P$ be the hypercube graph with vertex set $2^P$ (the power set of $P$) such that two sets $A,A'\subseteq P$ are adjacent if and only if $|A\triangle A'|=1$. For $S\subseteq 2^P$, let $\HH^P\vert_S$ be the induced subgraph of $\HH^P$ with vertex set $S$.

Let us now state our main results about irreducibility and stationary distributions of toggle Markov chains. As before, we fix a finite set $P$, a collection $\KK$ of subsets of $P$, an ordering ${\bf x}$ of the elements of $P$, and a probability $p_x$ for each $x\in P$.

\begin{theorem}\label{thm:toggle_irreducible}
Suppose $0<p_x<1$ for each $x\in P$. The toggle Markov chain $\TT(\KK,{\bf x})$ is irreducible if and only if the graph $\HH^P\vert_{\KK}$ is connected. 
\end{theorem}

If $P$ is a finite poset, then every connected component of $\HH^P_{J(P)}$ contains the empty set as a vertex. Thus, it is immediate from \Cref{thm:toggle_irreducible} that the rowmotion Markov chain ${\bf M}_{J(P)}$ is irreducible whenever $0<p_x<1$ for every $x\in P$. 

\begin{theorem}\label{thm:toggle_stationary}
Suppose that the toggle Markov chain $\TT(\KK,{\bf x})$ is irreducible and that $p_x>0$ for every $x\in P$. For $A\in \KK$, the probability of the state $A$ in the stationary distribution of $\TT(\KK,{\bf x})$ is \[\frac{1}{Z(\KK)}\prod_{x\in A}p_x^{-1},\] where $\displaystyle Z(\KK)=\sum_{A'\in \KK}\prod_{x'\in A'}p_{x'}^{-1}$.
\end{theorem}

Note that the stationary distribution in \Cref{thm:toggle_stationary} is independent of the ordering ${\bf x}$ (though the Markov chain itself can certainly depend on ${\bf x}$).

\subsection{Mixing Times}

Suppose ${\bf M}$ is an irreducible finite Markov chain with state space $\Omega$, transition matrix $Q$, and stationary distribution $\pi$. For $x\in\Omega$, let $Q^i(x,\cdot)$ denote the distribution on $\Omega$ in which the probability of a state $x'$ is the probability of reaching $x'$ by starting at $x$ and applying $i$ transitions (this probability is the entry in $Q^i$ in the row indexed by $x$ and the column indexed by $x'$).
In other words, $Q^i(x,\cdot)$ is the law on $\Omega$ after $i$ steps of the Markov chain, starting at $x$.
The \dfn{total variation distance} $\dTV=\dTV^{\Omega}$ is the metric on the space of distributions on $\Omega$ defined by
\[ \dTV(\mu,\nu)=\max_{A\subseteq \Omega} |\mu(A)-\nu(A)| = \frac{1}{2}\sum_{x\in\Omega} |\mu(x)-\nu(x)|. \]
For $\varepsilon>0$, the \dfn{mixing time} of $\bM$, denoted $\tmix_{\bM}(\varepsilon)$, is
the smallest nonnegative integer $i$ such that $\dTV(Q^i(x,\cdot),\pi)<\varepsilon$ for all $x\in\Omega$. 

The \dfn{width} of a finite poset $P$, denoted $\width(P)$, is the maximum size of an antichain in $P$. In \Cref{sec:mixing}, we use the method of coupling to prove the following bound on the mixing time of an arbitrary rowmotion Markov chain. 

\begin{theorem}\label{thm:general_mixing}
Let $P$ be a finite poset, and fix a probability $p_x\in(0,1)$ for each $x\in P$. Let $\overline p=\max\limits_{x\in P}p_x$. For each $\varepsilon>0$, the mixing time of $\bM_{J(P)}$ satisfies \[\tmix_{\bM_{J(P)}}(\varepsilon)\leq\left\lceil\frac{\log\varepsilon}{\log\left(1-\left(1-\overline{p}\right)^{\width(P)}\right)}\right\rceil.\]
\end{theorem}

We can drastically improve the bound in \Cref{thm:general_mixing} when $P$ is an antichain (so $J(P)$ is a Boolean lattice). For simplicity, we assume that all probabilities $p_x$ are equal to a single value $p$. In this setting, the Markov chain is reversible with respect to $\pi$; this allows us to give a spectral proof of the following result, which is an instance of the well-studied \emph{cutoff phenomenon}. (See \cite[Chapter~18]{Levin} for a discussion of cutoff.) 

\begin{theorem}\label{cutoff}
Let $P$ be an $n$-element antichain, and fix a probability $p\in(0,1)$. Let $p_x=p$ for all $x\in P$. Let $Q$ and $\pi$ be the transition matrix and stationary distribution, respectively, of the Markov chain ${\bf M}_{J(P)}$.  
\begin{enumerate}
    \item 
For $c>\frac{1}{2}$ and $t=\frac{1}{2} \log_{1/p}n + c$, we have
\[\max_{x\in J(P)}\dTV(Q^t(x,\cdot),\pi)\leq \frac{1}{2} \left( e^{p^{2c-1}} -1 \right)^{1/2}.\]
\item For $0<c<\frac{1}{2}\log_{1/p}n$ and $t=\frac{1}{2} \log_{1/p}n -c$, we have
\[\max_{x\in J(P)}\dTV(Q^t(x,\cdot),\pi) \geq 1- 4 p^{2c+1}-4 p^{2c}.\]
\end{enumerate}
\end{theorem}

It would be interesting to prove that other natural families of toggle Markov chains exhibit cutoff.

\subsection{Semidistrim Lattices}
If $P$ is a finite poset, then we can order $J(P)$ by inclusion to obtain a distributive lattice. In fact, Birkhoff's Fundamental Theorem of Finite Distributive Lattices \cite{Birkhoff} states that every finite distributive lattice is isomorphic to the lattice of order ideals of some finite poset. Thus, instead of viewing rowmotion as a bijective operator on the set of order ideals of a finite poset, one can equivalently view it as a bijective operator on the set of \emph{elements} of a distributive lattice. This perspective has led to more general definitions of rowmotion in recent years. Barnard \cite{Barnard} showed how to extend the definition of rowmotion to the broader family of \emph{semidistributive} lattices, while Thomas and Williams \cite{ThomasWilliams} discussed how to extend the definition to the family of \emph{trim} lattices. (Every distributive lattice is semidistributive and trim, but there are semidistributive lattices that are not trim and trim lattices that are not semidistributive.) 

One notable example motivating these extended definitions comes from Reading's \emph{Cambrian lattices} \cite{ReadingCambrian}. Suppose $c$ is a Coxeter element of a finite Coxeter group $W$. Reading \cite{ReadingClusters} found a bijection from the $c$-Cambrian lattice to the $c$-noncrossing partition lattice of $W$; under this bijection, rowmotion on the $c$-Cambrian lattice corresponds to the well-studied \emph{Kreweras complementation} operator on the $c$-noncrossing partition lattice of $W$ \cite{Barnard, ThomasWilliams}. See \cite{DefantLin, HopkinsCDE, ThomasWilliams} for other non-distributive lattices where rowmotion has been studied. 

Recently, the first author and Williams \cite{Semidistrim} introduced the even broader family of \emph{semidistrim} lattices and showed how to define a natural rowmotion operator on them; this is now the broadest family of lattices where rowmotion has been defined. It turns out that we can extend our definition of rowmotion Markov chains to semidistrim lattices; this provides a generalization of rowmotion Markov chains that is different from the toggle Markov chains discussed in \Cref{subsec:toggle}. Let us sketch the details here and wait until \Cref{sec:semidistrim} to define semidistrim lattices properly and explain why this definition specializes to the one given above when the lattice is distributive. 

Let $L$ be a semidistrim lattice, and let $\JJ_L$ and $\MM_L$ be the set of join-irreducible elements of $L$ and the set of meet-irreducible elements of $L$, respectively. There is a specific bijection $\kappa_L\colon\JJ_L\to\MM_L$ satisfying certain properties. The \emph{Galois graph} of $L$ is the loopless directed graph $G_L$ with vertex set $\JJ_L$ such that for all distinct $j,j'\in\JJ_L$, there is an arrow $j\to j'$ if and only if $j\not\leq\kappa_L(j')$. Let $\Ind(G_L)$ be the set of independent sets of $G_L$. There is a particular way to label the edges of the Hasse diagram of $L$ with elements of $\JJ_L$; we write $j_{uv}$ for the label of the edge $u\lessdot v$. For $w\in L$, let $\DD_L(w)$ be the set of labels of the edges of the form $u\lessdot w$, and let $\UU_L(w)$ be the set of labels of the edges of the form $w\lessdot v$. Then $\DD_L(w)$ and $\UU_L(w)$ are actually independent sets of $G_L$. Moreover, the maps $\DD_L,\UU_L\colon L\to \Ind(G_L)$ are bijections. The \emph{rowmotion} operator $\row\colon L\to L$ is defined by $\row=\UU_L^{-1}\circ\DD_L$. 

The \emph{rowmotion Markov chain} ${\bf M}_L$ has $L$ as its set of states. For each $j\in\JJ_L$, we fix a probability $p_j\in[0,1]$. Starting at a state $u\in L$, we choose a random subset $S$ of $\DD_L(u)$ by adding each element $j\in\DD_L(u)$ into $S$ with probability $p_j$ and then transition to the new state $u'=\row_L(\bigvee S)$. 

When $p_j=1$ for all $j\in\JJ_L$, the Markov chain $\bM_L$ is deterministic and agrees with rowmotion; indeed, this follows from \cite[Theorem~5.6]{Semidistrim}, which tells us that $\bigvee\DD_L(u)=u$ for all $u\in L$.  

Our main result about rowmotion Markov chains of semidistrim lattices is as follows. 

\begin{theorem}\label{thm:semidistrim_irreducible}
Let $L$ be a semidistrim lattice, and fix a probability $p_j\in(0,1)$ for each join-irreducible element $j\in\JJ_L$. The rowmotion Markov chain $\bM_L$ is irreducible. 
\end{theorem}

Let us remark that this theorem is not at all obvious. Our proof uses a delicate induction that relies on some difficult results about semidistrim lattices proven in \cite{Semidistrim}. For example, we use the fact that intervals in semidistrim lattices are semidistrim.

We can also generalize \Cref{thm:general_mixing} to the realm of semidistrim lattices in the following theorem. Given a semidistrim lattice $L$ and an element $u\in L$, we write $\ddeg(u)$ for the \dfn{down-degree} of $u$, which is the number of elements of $L$ covered by $u$. Let $\alpha(G_L)$ denote the independence number of the Galois graph $G_L$; that is, $\alpha(G_L)=\max\limits_{\mathcal I\in\Ind(G_L)}|\mathcal I|$. Equivalently, $\alpha(G_L)=\max\limits_{u\in L}\ddeg(u)$. If $P$ is a finite poset, then $\alpha(G_{J(P)})=\width(P)$.   

\begin{theorem}\label{thm:semidistrim_mixing}
Let $L$ be a semidistrim lattice, and fix a probability $p_j\in(0,1)$ for each $j\in \JJ_L$. Let $\overline p=\max\limits_{j\in\JJ_L}p_j$. For each $\varepsilon>0$, the mixing time of $\bM_{L}$ satisfies \[\tmix_{\bM_{L}}(\varepsilon)\leq\left\lceil\frac{\log\varepsilon}{\log\left(1-\left(1-\overline{p}\right)^{\alpha(G_L)}\right)}\right\rceil.\]
\end{theorem}

We were not able to find a formula for the stationary distribution of the rowmotion Markov chain of an arbitrary semidistrim (or even semidistributive or trim) lattice; this serves to underscore the anomalistic nature of the formula for distributive lattices in \eqref{eq:stationary}. However, there is one family of semidistrim (in fact, semidistributive) lattices where we were able to find such a formula. Given positive integers $a$ and $b$, let $\hexx_{a,b}$ be the lattice obtained by taking two disjoint chains $x_1<\cdots <x_a$ and $y_1<\cdots< y_b$ and adding a bottom element $\hat 0$ and a top element $\hat 1$. Let us remark that $\hexx_{m-1,m-1}$ is isomorphic to the weak order of the dihedral group of order $2m$, whereas $\hexx_{m-1,1}$ is isomorphic to the $c$-Cambrian lattice of that same dihedral group (for any Coxeter element $c$). We have $\mathcal J_{\hexx_{a,b}}=\mathcal M_{\hexx_{a,b}}=\{x_1,\ldots,x_a,y_1,\ldots,y_b\}$. For $2\leq i\leq a$ and $2\leq i'\leq b$, we have $\kappa_{\hexx_{a,b}}(x_i)=x_{i-1}$ and $\kappa_{\hexx_{a,b}}(y_{i'})=y_{i'-1}$; moreover, $\kappa_{\hexx_{a,b}}(x_1)=y_b$ and $\kappa_{\hexx_{a,b}}(y_1)=x_a$.  This is illustrated in \Cref{Fig3} when $a=3$ and $b=2$. \Cref{Fig2} shows the transition diagram of ${\bf M}_{\text{\hexx}_{2,1}}$. 

\begin{figure}[ht]
  \begin{center}{\includegraphics[height=8.992cm]{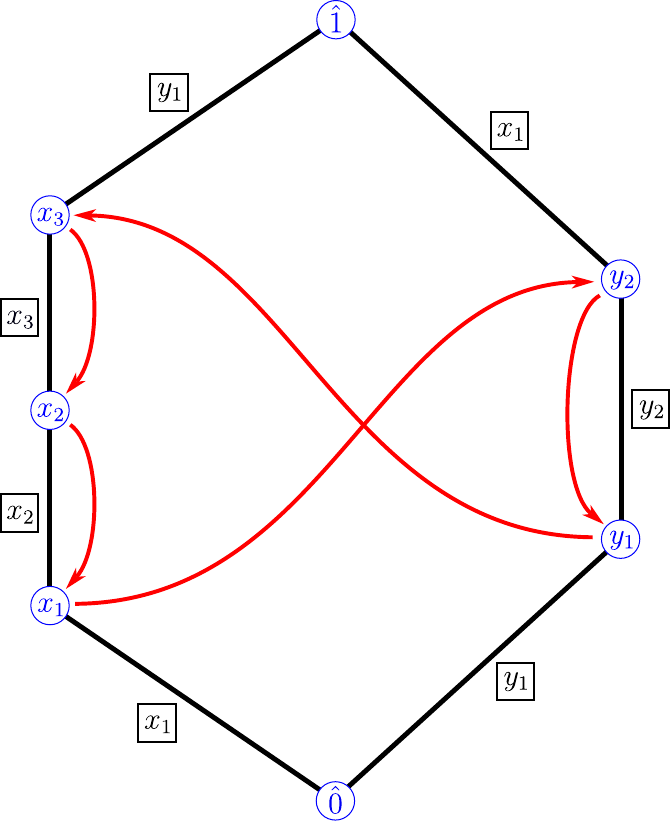}}
  \end{center}
  \caption{The lattice $\hexx_{3,2}$. Next to each edge $u\lessdot v$ is a box containing the edge label $j_{uv}$. The red arrows represent the action of $\kappa_{\hexx_{3,2}}$.}\label{Fig3}
\end{figure}

\begin{theorem}\label{thm:hexx}
Fix positive integers $a$ and $b$, and let $\kappa=\kappa_{\hexx_{a,b}}$. For each $j\in\JJ_{\text{\hexx}_{a,b}}$, fix a probability $p_j\in(0,1)$. There is a constant $Z(\hexx_{a,b})$ (depending only on $a$ and $b$) such that in the stationary distribution of $\bM_{\text{\hexx}_{a,b}}$, we have 
\begin{align*}
\mathbb P(\hat 0)&=\frac{1}{Z(\hexx_{a,b})}p_{x_1}p_{y_1}\left(1-\prod_{j\in\JJ_{\text{\hexx}_{a,b}}}p_j\right); \\
\mathbb P(\hat 1)&=\frac{1}{Z(\hexx_{a,b})}\left(1-\prod_{j\in\JJ_{\text{\hexx}_{a,b}}}p_j\right); \\ 
\mathbb P(x_i)&=\frac{1}{Z(\text{\hexx}_{a,b})}\left((1-p_{x_1})\prod_{\substack{j\in\JJ_{\text{\hexx}_{a,b}} \\ \kappa(j)\geq x_i}}p_j+(1-p_{y_1})\prod_{\substack{j\in\JJ_{\text{\hexx}_{a,b}} \\ \kappa(j)\not<x_i}}p_j\right)\quad\text{for}\quad 1\leq i \leq a; \\ 
\mathbb P(y_i)&=\frac{1}{Z(\text{\hexx}_{a,b})}\left((1-p_{y_1})\prod_{\substack{j\in\JJ_{\text{\hexx}_{a,b}} \\ \kappa(j)\geq y_i}}p_j+(1-p_{x_1})\prod_{\substack{j\in\JJ_{\text{\hexx}_{a,b}} \\ \kappa(j)\not<y_i}}p_j\right)\quad\text{for}\quad 1\leq i \leq b.  
\end{align*}
\end{theorem}

\begin{figure}[ht]
  \begin{center}{\includegraphics[height=11.702cm]{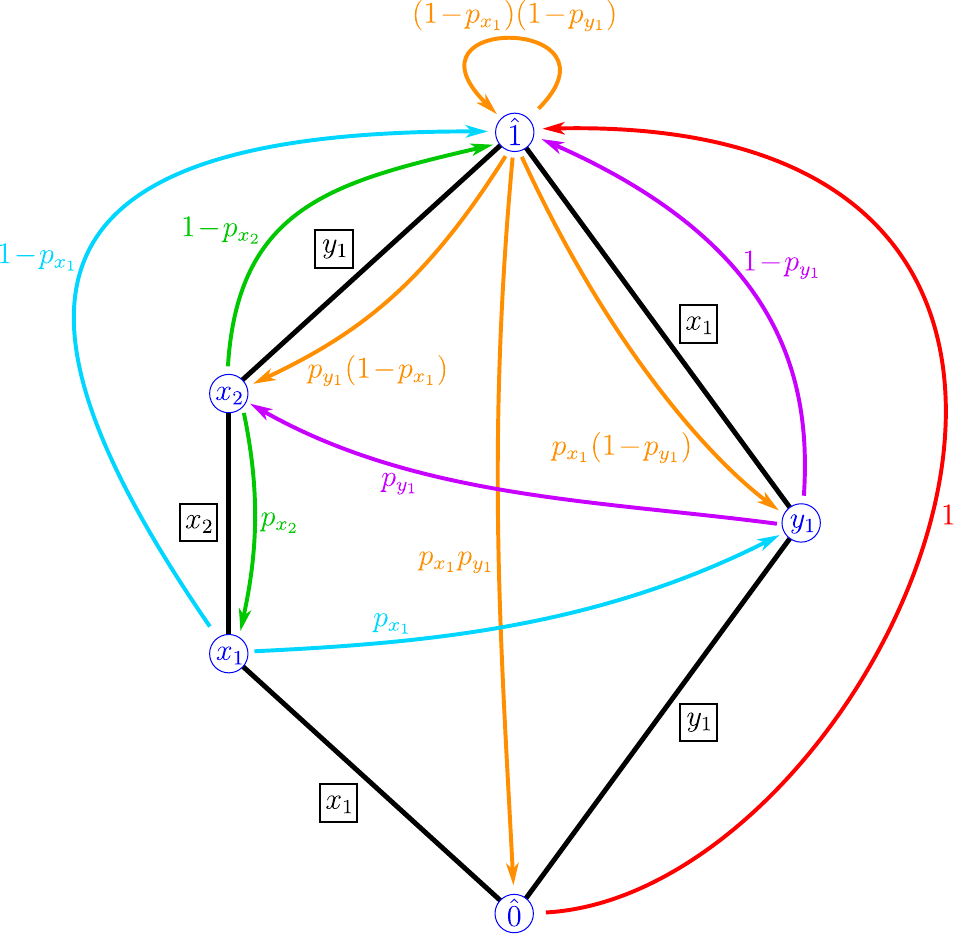}}
  \end{center}
  \caption{The transition diagram of $\bM_{\text{\hexx}_{2,1}}$ drawn over the Hasse diagram of $\hexx_{2,1}$. Next to each edge $u\lessdot v$ is a box containing the edge label $j_{uv}$. }\label{Fig2}
\end{figure}

\cref{sec:prelim} provides preliminary background on Markov chains and posets.
In \cref{sec:Toggle}, we prove \cref{thm:toggle_irreducible,thm:toggle_stationary}, which characterize when toggle Markov chains are irreducible and exhibit the stationary distributions of irreducible toggle Markov chains. In \cref{sec:semidistrim}, we recall how to define semidistrim lattices, define their rowmotion Markov chains, and prove \cref{thm:semidistrim_irreducible}, which states that such Markov chains are irreducible. \cref{sec:semidistrim} also contains the proof of \cref{thm:hexx}, which gives the stationary distribution of ${\bf M}_{\hexx_{a,b}}$. \cref{sec:mixing} is devoted to mixing times; it is in the section that we prove \cref{thm:general_mixing,cutoff}. We conclude in \cref{sec:conclusion} with a discussion of further research and open questions.

\section{Preliminaries}\label{sec:prelim}

\subsection{Markov Chains}

In this article, a (finite) \dfn{Markov chain} $\bM$ consists of a finite set $\Omega$ of \dfn{states} together with a \dfn{transition probability} $\mathbb P(s\to s')$ assigned to each pair $(s,s')\in\Omega\times\Omega$ so that $\sum_{s'\in\Omega}\mathbb P(s\to s')=1$ for every $s\in \Omega$. The set $\Omega$ is called the \dfn{state space} of $\bM$. We can represent $\bM$ via its \dfn{transition diagram}, which is the directed graph with vertex set $\Omega$ in which we draw an arrow $s\to s'$ labeled by the transition probability $\mathbb P(s\to s')$ whenever this transition probability is positive. We can also represent $\bM$ via its \dfn{transition matrix}, which is the matrix $Q=(Q(s,s'))_{s,s'\in\Omega}$ with rows and columns indexed by $\Omega$, where $Q(s,s')=\mathbb P(s\to s')$. Note that $Q$ is row-stochastic, meaning each of its rows consists of probabilities that sum to $1$. 

Say two states $s,s'\in\Omega$ \dfn{communicate} if there exist a directed path from $s$ to $s'$ and a directed path from $s'$ to $s$ in the transition diagram of $\bM$. There is an equivalence relation on $\Omega$ in which two states are equivalent if and only if they communicate; the equivalence classes are called \dfn{communicating classes}. We say $\bM$ is \dfn{irreducible} if there is exactly $1$ communicating class. 

A \dfn{stationary distribution} of $\bM$ is a probability distribution $\pi$ on $\Omega$ such that the vector $(\pi(s))_{s\in\Omega}$ is a left eigenvector of $Q$ with eigenvalue $1$. It is well known that if $\bM$ is irreducible, then it has a unique stationary distribution. 

\subsection{Posets}
All posets in this article are assumed to be finite. Given a poset $P$ and elements $x,y\in P$ with $x\leq y$, the \dfn{interval} from $x$ to $y$ is the set $[x,y]=\{z\in P:x\leq z\leq y\}$. Whenever we consider such an interval $[x,y]$, we will tacitly view it as a subposet of $P$. If $x<y$ and $[x,y]=\{x,y\}$, then we say $y$ \dfn{covers} $x$ and write $x\lessdot y$. 

A \dfn{lattice} is a poset $L$ such that any two elements $u,v\in L$ have a greatest lower bound, which is called their \dfn{meet} and denoted by $u\wedge v$, and a least upper bound, which is called their \dfn{join} and denoted by $u\vee v$. We denote the unique minimal element of $L$ by $\hat 0$ and the unique maximal element of $L$ by $\hat 1$. The meet and join operations are commutative and associative, so we can write $\bigwedge X$ and $\bigvee X$ for the meet and join, respectively, of an arbitrary subset $X\subseteq L$. We use the conventions $\bigwedge\emptyset=\hat 1$ and $\bigvee\emptyset=\hat 0$. An element that covers $\hat 0$ is called an \dfn{atom}.

\section{Irreducibility and Stationary Distributions of Toggle Markov Chains}\label{sec:Toggle}

In this section, we prove \Cref{thm:toggle_irreducible,thm:toggle_stationary}, which characterize when toggle Markov chains are irreducible (assuming each probability $p_x$ is strictly between $0$ and $1$) and provide the stationary distributions of irreducible toggle Markov chains, respectively. Recall the relevant notation and terminology from \cref{subsec:toggle}. 

\begin{proof}[Proof of \Cref{thm:toggle_irreducible}]
For each toggle operator $\tau_x$ and each set $A\in\KK$, the set $\tau_x(A)$ is either equal to $A$ or adjacent to $A$ in $\HH^P\vert_{\KK}$. Each transition in the Markov chain $\TT(\KK,{\bf x})$ is a composition of toggle operators. Therefore, if $\HH^P\vert_\KK$ is disconnected, then the Markov chain is not irreducible.  

We now prove the converse by induction on $n=|P|$. Assume $\HH^P\vert_\KK$ is connected. Suppose $A,A'\in\KK$ are nonempty sets that are adjacent in $\HH^P\vert_\KK$. We will show that there is a path from $A$ to $A'$ in the transition diagram of $\TT(\KK,{\bf x})$. Since $|A\triangle A'|=1$ and the sets $A$ and $A'$ are nonempty, there exists $z\in A\cap A'$. Let $\KK_z$ be the collection of all sets in $\KK$ that contain $z$. Let $\KK'$ be the vertex set of the connected component of $\HH^P\vert_{\KK_z}$ containing $A$ and $A'$. We can consider the toggle Markov chain $\TT(\KK',{\bf x})$. Let $\KK''=\{S\setminus\{z\}: S\in\KK'\}$. Let ${\bf x}'$ be the ordering of $P\setminus\{z\}$ obtained by deleting the element $z$ from ${\bf x}$. Since $\KK''$ is a collection of subsets of $P\setminus\{z\}$, we can consider the toggle Markov chain $\TT(\KK'',{\bf x}')$. The map $S\mapsto S\setminus \{z\}$ is an isomorphism from the connected graph $\HH^P\vert_{\KK'}$ to the graph $\HH^{P\setminus \{z\}}_{\KK''}$. This implies that $\HH^{P\setminus \{z\}}_{\KK''}$ is connected, so we can use induction to see that $\TT(\KK'',{\bf x}')$ is irreducible. Hence, there is a directed path from $A\setminus\{z\}$ to $A'\setminus\{z\}$ in the transition diagram of $\TT(\KK'',{\bf x}')$. The map $S\mapsto S\setminus\{z\}$ is also an isomorphism from the transition diagram of $\TT(\KK',{\bf x})$ to the transition diagram of $\TT(\KK'',{\bf x}')$, so there is a directed path from $A$ to $A'$ in the transition diagram of $\TT(\KK',{\bf x})$. This path is also present in the transition diagram of $\TT(\KK,{\bf x})$; indeed, whenever we apply the random toggle $\widetilde\tau_z$ to a set $B$, there is a positive probability (namely, $1-p_z$) that we do nothing and therefore keep $z$ in the set. 

It follows from the preceding paragraph that each connected component of $\HH^P\vert_{\KK\setminus\{\emptyset\}}$ is contained in a communicating class of $\TT(\KK,{\bf x})$. If $\emptyset\not\in \KK$, then this implies that $\TT(\KK,{\bf x})$ is irreducible. Let us now assume $\emptyset\in\KK$. It follows from the hypothesis that $\HH^P_\KK$ is connected that each connected component of $\HH^P\vert_{\KK\setminus\{\emptyset\}}$ contains a singleton set; hence, the proof will be complete if we can show that for every singleton set $\{x\}\in\KK$, there exist a directed path from $\{x\}$ to $\emptyset$ and a directed path from $\emptyset$ to $\{x\}$. Let us write ${\bf x}=(x_1,\ldots,x_n)$. Let $\{x_{i_1}\},\ldots,\{x_{i_r}\}$ be the singleton sets in $\KK$, where $i_1<\cdots<i_r$. Consider $j\in[r-1]$. Let \[B=(\tau_{x_1}\circ\cdots\circ\tau_{x_{i_{j}-1}})(\{x_{i_j}\})\quad\text{and}\quad B'=(\tau_{x_n}\circ\cdots\circ\tau_{x_{i_{j+1}+1}})(\{x_{i_{j+1}}\}).\] Note that $B$ is nonempty because it contains $x_{i_j}$ and that $B'$ is nonempty because it contains $x_{i_{j+1}}$. Thus, $B$ and $B'$ are in the connected components of $\HH^P\vert_{\KK\setminus\{\emptyset\}}$ containing $\{x_{i_j}\}$ and $\{x_{i_{j+1}}\}$, respectively. Because each connected component of $\HH^P\vert_{\KK\setminus\{\emptyset\}}$ is a communicating class of $\TT(\KK,{\bf x})$, there are directed paths from $\{x_{i_j}\}$ to $B$ and from $B'$ to $\{x_{i_{j+1}}\}$ in the transition diagram of $\TT(\KK,{\bf x})$. Also, because each toggle operator is an involution, we have $(\tau_{x_{i_j}}\circ\cdots\circ\tau_{x_1})(B)=\tau_{x_{i_j}}(\{x_{i_j}\})=\emptyset$ and $(\tau_{x_{i_{j+1}}}\circ\cdots\circ\tau_{x_{i_j+1}})(\emptyset)=\{x_{i_{j+1}}\}$. It follows that $\tau_P(B)=B'$, so there is a directed path from $B$ to $B'$ in the transition diagram of $\TT(\KK,{\bf x})$. Hence, there is a directed path from $\{x_{i_j}\}$ to $\{x_{i_{j+1}}\}$ in this transition diagram. A similar argument shows that there are directed paths from $\{x_{i_r}\}$ to $\emptyset$ and from $\emptyset$ to $\{x_{i_1}\}$. 
\end{proof}

\begin{proof}[Proof of \Cref{thm:toggle_stationary}]
For $S\in\KK$, let $\mu(S)=\prod\limits_{x\in S}p_x^{-1}$. Fix $A\in \KK$. We aim to show that $\sum\limits_{A'\in\KK}\mathbb P(A'\to A)\mu(A')=\mu(A)$.

Given a subset $U$ of $A$, let $\gamma(U)=\tau_{X\setminus U}^{-1}(A)$. We claim that $\gamma$ is a bijection from the collection of subsets of $A$ to the collection of sets $A'\in \KK$ such that $\mathbb P(A'\to A)>0$. It follows easily from the definition of the toggle operators that $\gamma(U)\cap A=U$. This implies that $\gamma$ is injective. It also shows that $U\subseteq\gamma(U)$. 

To see that $\gamma$ is surjective, suppose $A'\in\KK$ is such that $\mathbb P(A'\to A)>0$. Let $U=A\cap A'$. When we apply the sequence $\widetilde\tau_{x_1},\ldots,\widetilde\tau_{x_n}$ to $A'$ to obtain $A$, the set of elements $x$ such that the random toggle $\widetilde \tau_x$ does not apply the toggle $\tau_x$ is precisely $U$. This means that $A=\tau_{X\setminus U}(A')$, so $A'=\tau_{X\setminus U}^{-1}(A)=\gamma(U)$. Note also that $\mathbb P(A'\to A)=\prod\limits_{y\in A'\setminus U}p_y\prod\limits_{u\in U}(1-p_u)$. It follows that 
\begin{align*}
\sum_{A'\in\KK}\mathbb P(A'\to A)\mu(A')&=\sum_{U\subseteq A}\mathbb P(\gamma(U)\to A)\mu(\gamma(U)) \\ &=\sum_{U\subseteq A}\prod_{y\in \gamma(U)\setminus U}p_y\prod_{u\in U}(1-p_u)\prod_{x\in \gamma(U)}p_x^{-1} \\ 
&=\sum_{U\subseteq A}\prod_{u\in U}\frac{1-p_u}{p_u} \\&=\prod_{u\in A}\left(\frac{1-p_u}{p_u}+1\right) \\ &=\mu(A). \qedhere
\end{align*}
\end{proof}

\section{Semidistrim Lattices}\label{sec:semidistrim}

\subsection{Background}
This section follows \cite{Semidistrim}. Let $L$ be a lattice. An element $j\in L$ is called \dfn{join-irreducible} if it covers a unique element of $L$; in this case, we write $j_*$ for the unique element of $L$ covered by $j$.
Dually, an element $m\in L$ is called \dfn{meet-irreducible} if it is covered by a unique element of $L$; in this case, we write $m^*$ for the unique element of $L$ that covers $m$. Let $\JJ_L$ and $\MM_L$ be the set of join-irreducible elements of $L$ and the set of meet-irreducible elements of $L$, respectively. We say a join-irreducible element $j_0\in\JJ_L$ is \dfn{join-prime} if there exists $m_0\in\MM_L$ such that we have a partition $L=[j_0,\hat 1]\sqcup[\hat 0,m_0]$. In this case, $m_0$ is called \dfn{meet-prime}, and the pair $(j_0,m_0)$ is called a \dfn{prime pair}. 

A \dfn{pairing} on a lattice $L$ is a bijection $\kappa\colon \JJ_L\to \MM_L$ such that \[\kappa(j)\wedge j=j_*\quad\text{and}\quad\kappa(j)\vee j=(\kappa(j))^*\] for every $j\in \JJ_L$. (Not every lattice has a pairing.) We say $L$ is \dfn{uniquely paired} if it has a unique pairing; in this case, we denote the unique pairing by $\kappa_L$. If $L$ is uniquely paired and $(j_0,m_0)$ is a prime pair of $L$, then $\kappa_L(j_0)=m_0$.

Suppose $L$ is uniquely paired. For $u\in L$, we write \[J_L(u)=\{j\in\JJ_L:j\leq u\}\quad\text{and}\quad M_L(u)=\{j\in\JJ_L:\kappa_L(j)\geq u\}.\] There is an associated loopless directed graph $G_L$, called the \dfn{Galois graph} of $L$, defined as follows. The vertex set of $G_L$ is $\JJ_L$. For distinct $j,j'\in\JJ_L$, there is an arrow $j\to j'$ in $G_L$ if and only if $j\not\leq\kappa_L(j')$. An \dfn{independent set} of $G_L$ is a set $\mathcal I$ of vertices of $G_L$ such that for all $j,j'\in \mathcal I$, there is not an arrow from $j$ to $j'$ in $G_L$. Let $\Ind(G_L)$ be the set of independent sets of $G_L$. 

We say a uniquely paired lattice $L$ is \dfn{compatibly dismantlable} if either $|L|=1$ or there is a prime pair $(j_0,m_0)$ of $L$ such that the following compatibility conditions hold: 
\begin{itemize}
    \item $[j_0,\hat{1}]$ is compatibly dismantlable, and there is a bijection
    \[\alpha\colon M_L(j_0) \to \JJ_{[j_0,\hat 1]}\]
     given by $\alpha(j)=j_0\vee j$ such that $\kappa_{[j_0,\hat 1]}(\alpha(j))=\kappa_L(j)$ for all $j\in M_L(j_0)$;
    \item $[\hat{0},m_0]$ is compatibly dismantlable, and there is a bijection \[\beta\colon \kappa_L(J_L(m_0)) \to \MM_{[\hat 0,m_0]}\] given by $\beta(m)=m_0\wedge m$ such that $\beta(\kappa_L(j))=\kappa_{[\hat 0,m_0]}(j)$ for all $j\in J_L(m_0)$.
\end{itemize}
Such a prime pair $(j_0,m_0)$ is called a \dfn{dismantling pair} for $L$.

\begin{proposition}[{\cite[Proposition~5.3]{Semidistrim}}]
Let $L$ be a compatibly dismantlable lattice. For every cover relation $u\lessdot v$ in $L$, there is a unique join-irreducible element $j_{uv}\in J_L(v)\cap M_L(u)$.  
\end{proposition}

Suppose $L$ is compatibly dismantlable. The previous proposition allows us to label each edge $u\lessdot v$ in the Hasse diagram of $L$ with the join-irreducible element $j_{uv}$. For $w\in L$, we define the \dfn{downward label set} $\DD_L(w)=\{j_{uw}:u\lessdot w\}$ and the \dfn{upward label set} $\UU_L(w)=\{j_{wv}:w\lessdot v\}$. 

A lattice $L$ is called \dfn{semidistrim} if it is compatibly dismantlable and $\DD_L(w),\UU_L(w)\in\Ind(G_L)$ for all $w\in L$. As mentioned in \Cref{sec:intro}, semidistrim lattices generalize semidistributive lattices and trim lattices.

\begin{theorem}[{\cite[Theorem~6.2]{Semidistrim}}]
Semidistributive lattices are semidistrim, and trim lattices are semidistrim. Hence, distributive lattices are semidistrim. 
\end{theorem}

\begin{example}\label{exam:1}
Let us explicate how these general notions specialize when we consider a distributive lattice. Let $P$ be a finite poset, and let $L=J(P)$. For $x\in P$, write $\Delta(x)$ instead of $\Delta(\{x\})$ and $\nabla(x)$ instead of $\nabla(\{x\})$. There is a natural bijection $P\to \JJ_{L}$ given by $x\mapsto\Delta(x)$.
The unique pairing on $L$ is given by $\kappa_L(\Delta(x))=P\setminus\nabla(x)$. Every order ideal in $J(P)$ either contains $\Delta(x)$ or is contained in $P\setminus\nabla(x)$, but not both.
Hence, $(\Delta(x),P\setminus\nabla(x))$ is a prime pair (this shows that every join-irreducible element of a distributive lattice is join-prime). The Galois graph $G_L$ is isomorphic (via the map $x\mapsto\Delta(x)$) to the directed comparability graph of $P$, which has vertex set $P$ and has an arrow $x\to y$ for every strict order relation $y<x$ in $P$. The independent sets in $G_L$ correspond to antichains in $P$. 

It turns out that for any $x_0\in P$, the pair $(\Delta(x_0),P\setminus\nabla(x_0))$ is a dismantling pair. Indeed, the interval $[\Delta(x_0),\hat 1]=\{I\in J(P):x_0\in I\}$ can be identified with the lattice $J(P\setminus\Delta(x_0))$, and the set $M_L(\Delta(x_0))$ can be identified with $P\setminus\Delta(x_0)$. Hence, the bijection $\alpha\colon M_L(\Delta(x_0))\to\JJ_{[\Delta(x_0),\hat 1]}$ is the usual correspondence between elements of $P\setminus\Delta(x_0)$ and join-irreducible elements of $J(P\setminus\Delta(x_0))$. Similarly, $[\hat 0,P\setminus\nabla(x_0)]=\{I\in J(P):x_0\not\in I\}$ can be identified with the lattice $J(P\setminus\nabla(x_0))$, and the set $\kappa_L(J_L(P\setminus\nabla(x_0)))$ can be identified with $P\setminus\nabla(x_0)$. Hence, the bijection $\beta\colon \kappa_L(J_L(P\setminus\nabla(x_0)))\to\MM_{[\hat 0,P\setminus\nabla(x_0)]}$ is the usual correspondence between elements of $P\setminus\nabla(x_0)$ and meet-irreducible elements of $J(P\setminus\nabla(x_0))$. Whenever we have a cover relation $I\lessdot I'$ in $L$, there is a unique element $z\in P$ such that $I'=I\sqcup\{z\}$; then $j_{II'}=\Delta(z)$. For $I\in L$, the downward label set $\DD_L(I)$ and the upward label set $\UU_L(I)$ correspond (via the map $x\mapsto\Delta(x)$) to $\max(I)$ and $\min(P\setminus I)$, respectively; these are both independent sets in $G_L$ (i.e., antichains in $P$), so $L$ is semidistrim. 
\end{example}

\begin{theorem}[{\cite[Theorem~6.4]{Semidistrim}}]\label{thm:row_well_defined}
If $L$ is a semidistrim lattice, then the maps $\DD_L\colon L\to\Ind(G_L)$ and $\UU_L\colon L\to\Ind(G_L)$ are bijections. 
\end{theorem}

Let $L$ be a semidistrim lattice. Using the preceding theorem, we can define the \dfn{rowmotion} operator $\row_L\colon L\to L$ by \[\row_L=\UU_L^{-1}\circ\DD_L.\] Referring to \Cref{exam:1}, we find that this definition agrees with the one given in \eqref{eq:row_def} when $L$ is distributive. Moreover, this definition coincides with the definition due to Barnard \cite{Barnard} when $L$ is semidistributive and with the definition due to Thomas and Williams \cite{ThomasWilliams} when $L$ is trim. 

\subsection{Rowmotion Markov Chains on Semidistrim Lattices}

We can now define rowmotion Markov chains on semidistrim lattices, thereby providing another generalization of the definition we gave in \Cref{subsec:distributive} for distributive lattices. 

\begin{definition}\label{def:semidistrim_Markov}
Let $L$ be a semidistrim lattice. For each $j\in \JJ_L$, fix a probability $p_j\in[0,1]$. We define the \dfn{rowmotion Markov chain} $\bM_L$ as follows. The state space of $\bM_L$ is $L$. For any $u,u'\in L$, the transition probability from $u'$ to $u$ is \[\mathbb P(u'\to u)=\begin{cases} \left(\prod\limits_{j\in \UU_L(u)}p_j\right)\left(\prod\limits_{j'\in \DD_L(u')\setminus\UU_L(u)}(1-p_{j'})\right) & \mbox{if }\UU_L(u)\subseteq\DD_L(u'); \\ 0 & \mbox{otherwise.}\end{cases}\]
\end{definition}

If $L$ is semidistrim and $u\in L$, then we have \cite[Theorem~5.6]{Semidistrim} \[u=\bigvee\DD_L(u)=\bigwedge\kappa_L(\UU_L(u)).\]
It follows from \Cref{thm:row_well_defined} that $\mathcal I=\DD_L(\bigvee\mathcal I)=\UU_L(\bigwedge\kappa_L(\mathcal I))$ for every $\mathcal I\in\Ind(G_L)$. 
This enables us to give a more intuitive description of the rowmotion Markov chain $\bM_L$ as follows.
Starting from a state $u\in L$, choose a random subset $S$ of $\DD_L(u)\in\Ind(G_L)$ by adding each element $j\in\DD_L(u)$ into $S$ with probability $p_j$; then transition to the new state $\bigwedge\kappa_L(S)=\row_L(\bigvee S)$.
Observe that if $p_j=1$ for all $j\in\JJ_L$, then $\bM_{L}$ is deterministic and agrees with rowmotion. On the other hand, if $p_j=0$ for all $j\in \JJ_L$, then $\bM_L$ is deterministic and sends all elements of $L$ to~$\hat 1$.

\subsection{Irreducibility}\label{sec:irreducible}
In this subsection, we prove \Cref{thm:semidistrim_irreducible}, which tells us that rowmotion Markov chains of semidistrim lattices are irreducible. When $L$ is a distributive lattice, this follows from \Cref{thm:toggle_irreducible}. To handle arbitrary semidistrim lattices, we need a different strategy that utilizes the following difficult result from \cite{Semidistrim}.  

\begin{theorem}[{\cite[Theorem~7.8, Corollary~7.9, Corollary~7.10]{Semidistrim}}]\label{thm:stuff_we_need}
Let $L$ be a semidistrim lattice, and let $[u,v]$ be an interval in $L$. Then $[u,v]$ is a semidistrim lattice. There are bijections \[\alpha_{u,v}\colon J_L(v)\cap M_L(u)\to\JJ_{[u,v]}\quad\text{and}\quad\beta_{u,v}\colon\kappa_L(J_L(v)\cap M_L(u))\to\MM_{[u,v]}\] given by $\alpha_{u,v}(j)=u\vee j$ and $\beta_{u,v}(m)=v\wedge m$. We have $\kappa_{[u,v]}(\alpha_{u,v}(j))=\beta_{u,v}(\kappa_L(j))$ for all $j\in J_L(v)\cap M_L(u)$. The map $\alpha_{u,v}$ is an isomorphism from an induced subgraph of the Galois graph $G_L$ to the Galois graph $G_{[u,v]}$. If $u\leq w\lessdot w'\leq v$ and $j_{ww'}$ is the label of the cover relation $w\lessdot w'$ in $L$, then $\alpha_{u,v}(j_{ww'})$ is the label of the same cover relation in $[u,v]$.  
\end{theorem}

\begin{proof}[Proof of \Cref{thm:semidistrim_irreducible}]
Let $L$ be a semidistrim lattice. The proof is trivial when $|L|=1$, so we may assume $|L|\geq 2$ and proceed by induction on $|L|$. Fix $u\in L$. The transition diagram of $\bM_L$ contains an arrow $u\to \hat 1$; our goal is to prove that it also contains a path from $\hat 1$ to $u$. 

First, suppose $u=\hat 0$. Let $k$ be the size of the orbit of $\row_L$ containing $\hat 1$. The transition diagram of $\bM_L$ contains the path \[\hat 1\to\row_L(\hat 1)\to\row_L^2(\hat 1)\to\cdots\to\row_L^{k-1}(\hat 1).\] But $\row_L(\hat 0)=\hat 1=\row_L(\row_L^{k-1}(\hat 1))$, so $\row_L^{k-1}(\hat 1)=\hat 0=u$. This demonstrates that the desired path exists in this case.

Now suppose $u\neq\hat 0$. By \Cref{thm:stuff_we_need}, the interval $[u,\hat 1]$ is a semidistrim lattice, so it follows by induction that $\bM_{[u,\hat 1]}$ is irreducible. Suppose $w\to w'$ is an arrow in the transition diagram of $\bM_{[u,\hat 1]}$. This means that there exists a set $S\subseteq \DD_{[u,\hat 1]}(w)$ such that $w'=\bigwedge\kappa_{[u,\hat 1]}(S)$. Let $\alpha_{u,\hat 1}$ and $\beta_{u,\hat 1}$ be the bijections from \Cref{thm:stuff_we_need} (where we have set $v=\hat 1$). Note that $\beta_{u,\hat 1}(m)=\hat 1\wedge m=m$ for all $m\in\kappa_L(M_L(u))$. Let
$T=\alpha_{u,\hat 1}^{-1}(S)$. It follows from \Cref{thm:stuff_we_need} that $T\subseteq\DD_L(w)$ and that \[w'=\bigwedge\kappa_{[u,\hat 1]}(S)=\bigwedge\kappa_{[u,\hat 1]}(\alpha_{u,\hat 1}(T))=\bigwedge\beta_{u,\hat 1}(\kappa_L(T))=\bigwedge\kappa_L(T).\] This shows that $w\to w'$ is an arrow in the transition diagram of $\bM_{L}$. 

We have proven that all arrows in the transition diagram of $\bM_{[u,\hat 1]}$ are also arrows in the transition diagram of $\bM_L$. Since $\bM_{[u,\hat 1]}$ is irreducible, there is a path from $\hat 1$ to $u$ in the transition diagram of $\bM_{[u,\hat 1]}$. This path is also in the transition diagram of $\bM_{L}$, so the proof is complete. 
\end{proof}

\subsection{A Special Class of Semidistrim Lattices}\label{sec:stationary}

This subsection is devoted to proving \cref{thm:hexx}.
That is, we will compute the stationary distribution of the rowmotion Markov chain of $\hexx_{a,b}$.

\begin{proof}[Proof of \Cref{thm:hexx}]
Recall that $\hexx_{a,b}$ is obtained by adding the minimal element $\hat 0$ and the maximal element $\hat 1$ to the disjoint chains $x_1<\cdots<x_a$ and $y_1<\cdots<y_b$. The join-irreducibles are $x_1,\dots,x_a,y_1,\dots,y_b$; for simplicitly, let $q_i=p_{x_i}$ and $r_i = p_{y_i}$.
The transition probabilities of ${\bf M}_{\text{\hexx}_{a,b}}$ are as follows:
\begin{equation*}
    \setlength{\tabcolsep}{12pt}
    \begin{tabular}{llll}
        $\P(\hat 1 \to \hat 0) = q_1 r_1$ & $\P(\hat 0 \to \hat 1) = 1$ & $\P(x_{i+1}\to x_i) = q_{i+1}$ & $\P(y_{i+1} \to y_i) = r_{i+1}$ \\
        $\P(\hat 1 \to \hat 1) = (1-q_1)(1-r_1)$ & & $\P(x_1 \to y_b) = q_1$ & $\P(y_1 \to x_a) = r_1$ \\
        $\P(\hat 1 \to x_a) = (1-q_1)r_1$ & & $\P(x_i \to \hat 1) = 1-q_i$ & $\P(y_i \to \hat 1) = 1-r_i$ \\
        $\P(\hat 1 \to y_b) = q_1(1-r_1)$. &&& 
    \end{tabular}
\end{equation*}
Removing the $Z(\hexx_{a,b})$ normalization factor, it suffices to show the following measure $\mu$ is stationary:
\begin{align*}
    \mu(\hat 0) &= q_1r_1\left(1-\prod_{i=1}^a q_i \prod_{i'=1}^b r_{i'}\right) \\
    \mu(\hat 1) &= 1-\prod_{i=1}^a q_i \prod_{i'=1}^b r_{i'} \\
    \mu(x_i) &= (1-q_1) r_1 \prod_{k=i+1}^a q_k + q_1(1-r_1) \prod_{k=i+1}^a q_k \prod_{k'=1}^b r_{k'} \\
    \mu(y_i) &= q_1(1-r_1) \prod_{k=i+1}^b r_k + (1-q_1)r_1\prod_{k=1}^a q_j \prod_{k'=i+1}^b r_{k'}.
\end{align*}
First, we have
\begin{align*}
    \sum_{z\in\,\hexx_{a,b}}\mu(z)\P(z\to\hat0)
    = q_1r_1\mu(\hat 1) &= \mu(\hat 0); \\
    \sum_{z\in\,\hexx_{a,b}}\mu(z)\P(z\to x_i)
    = q_{i+1}\mu(x_{i+1}) &= \mu(x_i); \\
    \sum_{z\in\,\hexx_{a,b}}\mu(z)\P(z\to y_j) = r_{i'+1}\mu(y_{i'+1}) &= \mu(y_{i'})
\end{align*}
for $1 \leq i \leq a-1$ and $1 \leq i' \leq b-1$.
For $\hat 1$, we have
\begin{align*}
    \sum_{z\in\,\hexx_{a,b}}\mu(z)\P(z\to\hat1)
    &= (1-q_1)(1-r_1)\mu(\hat1) + \mu(\hat0) + \sum_{i=1}^a (1-q_i)\mu(x_i) + \sum_{i'=1}^b (1-r_{i'})\mu(y_{i'}).
\end{align*}
To show this equals $\mu(\hat1)$, it suffices to show
\begin{align*}
    \sum_{i=1}^a (1-q_i)\mu(x_i) + \sum_{i'=1}^b (1-r_{i'})\mu(y_{i'}) = (q_1 + r_1 - 2q_1r_1)\mu(\hat1).
\end{align*}
We expand the first sum as
\begin{align*}
    \sum_{i=1}^a (1-q_i)\mu(x_i)
    &= (1-q_1)r_1 \sum_{i=1}^a \left((1-q_i)\prod_{k=i+1}^a q_k\right) + q_1(1-r_1)\prod_{i'=1}^b r_{i'} \sum_{i=1}^a \left((1-q_i)\prod_{k=i+1}^a q_k\right)
    \\ &= (1-q_1)r_1\left(1-\prod_{i=1}^a q_i\right) + q_1(1-r_1)\prod_{i'=1}^b r_{i'}\left(1-\prod_{i=1}^a q_i\right)
    \\ &= (1-q_1)r_1 + q_1(1-r_1)\prod_{i'=1}^b r_{i'} - (1-q_1)r_1\prod_{i=1}^a q_i - q_1(1-r_1)\prod_{i=1}^a q_i \prod_{i'=1}^b r_{i'}
\end{align*}
and similarly expand the second sum as
\begin{align*}
    \sum_{i'=1}^b (1-r_{i'})\mu(y_{i'})
    = q_1(1-r_1) + (1-q_1)r_1\prod_{i=1}^a q_i -q_1(1-r_1)\prod_{i'=1}^b r_{i'}- (1-q_1)r_1\prod_{i=1}^a q_i \prod_{i'=1}^b r_{i'}.
\end{align*}
Combining them yields
\begin{align*}
    \sum_{i=1}^a (1-q_i)\mu(x_i) + \sum_{i=1}^b (1-r_i)\mu(y_i)
    &= \left((1-q_1)r_1 + q_1(1-r_1)\right)\left(1-\prod_{i=1}^a q_i\prod_{i'=1}^b r_{i'}\right)
    \\ &= (q_1+r_1-2q_1r_1)\mu(\hat 1),
\end{align*}
as desired.
It remains to check $x_a$ and $y_b$.
For $x_a$, we have
\begin{align*}
    \sum_{z\in\,\hexx_{a,b}}\mu(z)\P(z\to x_a)
    &= (1-q_1)r_1\mu(\hat1) + r_1\mu(y_1)
    \\ &= (1-q_1)r_1 - (1-q_1)r_1 \prod_{i=1}^a q_i \prod_{i'=1}^b r_{i'}  \\ &\hphantom{=}+ q_1(1-r_1)\prod_{i'=1}^b r_{i'} + (1-q_1)r_1 \prod_{i=1}^a q_i \prod_{i'=1}^b r_{i'}
    \\ &= (1-q_1)r_1 + q_1(1-r_1)\prod_{i'=1}^b r_{i'} \\ 
    &= \mu(x_a).
\end{align*}
The computation for $y_b$ is essentially identical. This completes the proof that $\mu$ is stationary.
\end{proof}

\section{Mixing Times}\label{sec:mixing}
We now study the mixing times of rowmotion Markov chains.

\subsection{Couplings}
Let ${\bf M}$ be an irreducible Markov chain with state space $\Omega$, stationary distribution $\pi$, and transition probabilities $\mathbb P(s\to s')$ for all $s,s'\in \Omega$. A \dfn{Markovian coupling} for ${\bf M}$ is a sequence $(X_i,Y_i)_{i\geq 0}$ of pairs of random variables with values in $\Omega$ such that for every $i\geq 0$ and all $s,s',s''\in\Omega$, we have \[\mathbb P(X_{i+1}=s\vert X_i=s', Y_i=s'')=\mathbb P(s'\to s)\] and \[\mathbb P(Y_{i+1}=s\vert X_i=s', Y_i=s'')=\mathbb P(s''\to s).\] 
It is well known that
\begin{equation}
    \dTV(Q^i(x,\cdot),\pi) \leq \P(X_i\neq Y_i)
\end{equation}
for any Markovian coupling for ${\bf M}$ with $X_0=x$ and $Y_0\sim\pi$.
In fact, for a \dfn{coupling} of two distributions $\mu$ and $\nu$ on $\Omega$, i.e., a joint distribution $(X,Y)$ with marginal distributions $X\sim\mu$ and $Y\sim\nu$, we have
\begin{equation}
   \dTV(\mu,\nu) \leq \P(X\neq Y),
\end{equation}
and this inequality becomes equality when taking the infimum of $\P(X\neq Y)$ over all such couplings $(X,Y)$.
A coupling $(X,Y)$ such that $\dTV(\mu,\nu)=\P(X\neq Y)$ always exists and is called an \dfn{optimal} coupling.

\subsection{General Upper Bound}
We now prove \cref{thm:general_mixing}.
\begin{proof}[Proof of \cref{thm:general_mixing}]
Recall that $\overline p = \max\limits_{x\in P} p_x$.
For any $I\in J(P)$, we have (viewing $P$ as an element of $J(P)$)
\[ \mathbb P(I\to P)=\prod_{x\in\max(I)} (1-p_x) \geq (1-\overline p)^{\lvert\max(I)\rvert} \geq (1-\overline p)^{\width(P)}. \]
Thus, for any $I\in J(P)$, we can construct a Markovian coupling $(X_i,Y_i)_{i\geq 0}$ with $X_0=I$ and $Y_0\sim\pi$ such that
\begin{itemize}
\item if $X_i\neq Y_i$, then $X_{i+1}=Y_{i+1}=P$ with probability at least $(1-\overline p)^{\width(P)}$ (the other transition probabilities do not matter so long as they induce the correct marginal transition probabilities) and 
\item if $X_i=Y_i$, then $X_{i+1}=Y_{i+1}$.
\end{itemize}
This implies
\[ \P(X_{i+1}\neq Y_{i+1})\leq \left(1-(1-\overline p)^{\width(P)}\right)\P(X_i\neq Y_i)\] for all $i\geq 0$, so
\[ \dTV(Q^k(x,\cdot),\pi) \leq \P(X_k\neq Y_k) \leq \left(1-(1-\overline p)^{\width(P)}\right)^k \]
for all $k \geq 0$.
As the inequality
\[ \left(1-(1-\overline p)^{\width(P)}\right)^k \leq \varepsilon \]
is equivalent to
\[ i \geq \frac{\log\varepsilon}{\log\left(1-(1-\overline p)^{\width(P)}\right)}, \]
we have
\[ \tmix_{\bM_{J(P)}}(\varepsilon) \leq \left\lceil \frac{\log\varepsilon}{\log\left(1-(1-\overline p)^{\width(P)}\right)}\right\rceil, \]
as desired.
\end{proof}
\begin{proof}[Proof of \cref{thm:semidistrim_mixing}]
    Let $L$ be a semidistrim lattice, and consider the Markov chain ${\bf M}_L$. For each $u\in L$, we have
    \[ \mathbb P(u\to\hat 1)=\prod_{j\in\DD_L(u)} (1-p_j) \geq (1-\overline p)^{|\DD_L(u)|}\geq (1-\overline p)^{\alpha(G_L)}.\]
    The rest of the proof then follows just as in the preceding proof of \Cref{thm:general_mixing}.
\end{proof}
\begin{remark}\label{rem:improve_mixing}
    Let $\mathcal A(P)$ be the set of antichains of a poset $P$. We can straightforwardly improve the $\log\left(1-(1-\overline p)^{\width(P)}\right)$ term of \cref{thm:general_mixing} by instead using
    \[ \log\left(1-\min_{A\in\mathcal A(P)}\prod_{x\in A}(1-p_x)\right); \]
    however, when $p_x=p$ is the same across all $x\in P$, or more generally when some antichain $A$ of size $|A|=\width(P)$ has $p_x=\overline p$ for all $x\in A$, these two bounds coincide.
Similarly, we can improve the $\log\left(1-(1-\overline p)^{\alpha(G_L)}\right)$ term of \cref{thm:semidistrim_mixing} by instead using
    \[ \log\left(1-\min_{\mc I\in\Ind(G_L)}\prod_{j\in\mc I} (1-p_j)\right). \]
\end{remark}

\subsection{Boolean Lattices}\label{subsec:boolean}
In this subsection, we present the proof of \cref{cutoff}. Let $P$ be an $n$-element antichain. 
Fix a probability $p\in(0,1)$, and let $p_x=p$ for all $x\in P$. The set of states of ${\bf M}_{J(P)}$ is $\Omega=2^P$, the power set of $P$. For $I,I'\subseteq P$, we have \[\mathbb P(I\to I')=\begin{cases} p^{|P\setminus I'|}(1-p)^{|I\cap I'|} & \mbox{if }P\setminus I'\subseteq I; \\ 0 & \mbox{otherwise.}\end{cases}\]
Let $Q$ denote the transition matrix of ${\bf M}_{J(P)}$. Let $\pi$ be the stationary distribution of ${\bf M}_{J(P)}$, which we computed explicitly in \cref{thm:toggle_stationary}. 

We begin by discussing the spectrum of $Q$. For $I\subseteq P$, define $f_I:2^P \to \mathbb{C}$ by
$f_I (A)= p^{-|I|/2}(-p)^{|I\cap A|}$. 

\begin{lemma}\label{orth}
The eigenvalues of $Q$ are the numbers $(-p)^k$ for $0\leq k\leq n$. A basis for the eigenspace of $Q$ with eigenvalue $(-p)^k$ is $\{f_I:I\subseteq[n], |I|=k\}$. Moreover, the basis $\{f_I:I\subseteq P\}$ of eigenvectors of $Q$ is orthonormal with respect to $\pi$.  
\end{lemma}
 \begin{proof}
 For $I,A\subseteq P$, we have
     \begin{align*}Qf_I(A)&=\sum_{A'\subseteq P} Q(A,A')f_I(A') \\
     &= p^{-|I|/2}\sum_{j=0}^{|I\cap A|} {|I\cap A| \choose j}(-p)^{|I|- |I\cap A|+j} (1-p)^jp^{|I\cap A|-j}\\
     &= p^{-|I|/2}(-p)^{|I|} \sum_{j=0}^{|I\cap A|} {|I\cap A| \choose j} (1-p)^j(-1)^{|I\cap A|-j}\\
     &= p^{-|I|/2}(-p)^{|I|} (-p)^{|I\cap A|}\\ 
     &= (-p)^{|I|}f_I(A).
     \end{align*}
 Thus, $f_I$ is an eigenvector of $Q$ with eigenvalue $(-p)^{|I|}$.

Now, fix $I,J\subseteq P$ with $I\neq J$. We have
\begin{align*}
\sum_{ A \subseteq P} (f_I(A))^2 \pi(A)&= p^{-|I|} \sum_{ A \subseteq P}  p^{ \vert I\cap A \vert} \frac{p^{-|A|}}{\left(1+ \frac{1}{p} \right)^n}\\
&= p^{-|I|} \left(1+ \frac{1}{p} \right)^{-n} \sum_{j=0}^{|I|}\sum_{i=0}^{n-|I|}{|I| \choose j}{n-|I| \choose i} p^{2j-j-i}\\
&= p^{-|I|} \left(1+ \frac{1}{p} \right)^{-n} \left(1+ \frac{1}{p} \right)^{n-|I|} (1+p)^{|I|}=1.
\end{align*}
Moreover,  
\begin{align*}
\sum_{ A \subseteq P} f_I(A)f_J(A) \pi(A)&= \left(1+ \frac{1}{p} \right)^{-n}  p^{-\frac{|I|+|J|}{2}} \sum_{ A \subseteq P}  p^{\vert J\cap A \vert}p^{ \vert I\cap A \vert} p^{-|A|}. 
\end{align*}
This last sum can be written as
\begin{align*}
\sum_{k=0}^{|I\cap J|}\sum_{i=0}^{|I|-|I\cap J|}\sum_{j=0}^{|J|-|I\cap J|}\sum_{\ell=0}^{n-|I\cup J|}\textstyle{|I \cap J| \choose k}{|I|- |I \cap J| \choose i}{|J|- |I \cap J| \choose j}{n- |I \cup J| \choose \ell} (-p)^{2k+j+i}p^{-(k+i+j+\ell)},
\end{align*}
and this is zero because either $\sum\limits_{j=0}^{|J|-|I\cap J|}(-1)^j{|J|- |I \cap J| \choose j}$ or $\sum\limits_{i=0}^{|I|-|I\cap J|}(-1)^j{|I|- |I \cap J| \choose i}$ is zero (since $I\neq J$).
\end{proof}

We now proceed to prove the inequalities in \cref{cutoff}. We begin with the upper bound on the total variation distance (which corresponds to an upper bound on the mixing time).

\begin{proof}[Proof of \cref{cutoff}, Part (1)] We will use the standard $\ell_2$ inequality (see \cite[Lemma~12.18]{Levin})
\[4\dTV(Q^t(x,\cdot),\pi)^2 \leq \sum_{\emptyset \neq I \subseteq P} (f_I(x))^2 ((-p)^{|I|})^{2t},\]
which holds for every $x \subseteq P$.  

Using \Cref{orth} and the fact that $\Vert f_I^2\Vert_{\infty} \leq p^{-|I|}$, we get
\begin{align}
4\dTV(Q^t(x,\cdot),\pi)^2 &\leq \sum_{\emptyset \neq I \subseteq P} p^{-|I|} p^{2|I|t} \cr
 & \leq \sum_{j=1}^n {n \choose j} p^{-j}p^{2jt} \cr
 & \leq \label{beg} \sum_{j=1}^n \frac{n^j}{j!} p^{-j}p^{2jt} .  
\end{align}
For $t=\frac{1}{2} \log_{1/p}n + c$,
\[
 \sum_{j=1}^n \frac{n^j}{j!} p^{-j}p^{2jt}\leq \sum_{j=1}^n \frac{p^{(2c-1)j}}{j!} \leq \left( e^{p^{2c-1}}-1\right). \qedhere  
\]
\end{proof}

We now proceed to prove the lower bound on the total variation distance in \cref{cutoff}. 

The state of ${\bf M}_{J(P)}$ at time $t$ is a subset of $P$; let $X_t$ denote the size of this state. Define functions $f,g\colon\mathbb R\to\mathbb R$ by \[f(x)= 1-\frac{(1+p)x}{n}\quad\text{and}\quad g(x)=-\frac{p+1}{n}x^2 +\frac{p+2n-1}{n}x - \frac{n-1}{p+1}.\]
\begin{lemma}\label{expectation}
We have
\[\expect{f(X_{t+1}) \vert X_t }=-p f(X_t) \quad\text{and}\quad\expect{g(X_{t+1}) \vert X_t }=p^2 g(X_t) . \]
\end{lemma}
\begin{proof}
We have \[\expect{X_{t+1} \vert X_t }= n-X_t +(1-p)X_t= n-pX_{t}. \]
Therefore,
 \[\expect{f(X_{t+1}) \vert X_t }= 1- \frac{(1+p)(n-pX_t)}{n}=-pf(X_t). \]
Now, using the fact that 
\[\expect{X_{t+1}^2 \vert X_t=s }= \sum_{i=0}^s (n-s +i)^2 {s \choose i} p^{s-i}(1-p)^i= p^2s^2-(p^2+2np-p)s+n^2, \]
we can easily check that
\[\expect{g(X_{t+1}) \vert X_t }=p^2 g(X_t) . \qedhere \]
\end{proof}                                                                                        
The main observation that allows us to compute the variance of $f(X_t)$ is the polynomial identity
\begin{equation}\label{poly}
    f^2=- \frac{1+p}{n}g +\frac{1-p}{n}f+\frac{p}{n}.
\end{equation}
The next lemma discusses the variances of $f(X_t)$ and $f(X)$, where $X$ is the size of a set that is distributed according to the stationary measure $\pi$. 
\begin{lemma}\label{variance}
We have
\[\varn(f(X_t)\vert X_0=0)= \frac{p}{n}-\frac{1}{n} p^{2t} +\frac{1-p}{n}(-p)^t  \]
and
\[\varn(f(X))= \frac{p}{n}. \]
\end{lemma}
\begin{proof}
For the first equation, we write \[\varn(f(X_t)\vert X_0=0)= \expect{f(X_t)^2\vert X_0=0}-\left(\expect{f(X_t)\vert X_0=0}\right)^2\] and use \eqref{poly} and then \Cref{expectation} $t$ times. To compute $\varn(f(X))$, we use the identity \eqref{poly} and the fact that $\expect{f(X)}=\expect{g(X)}=0$, which follows from \cref{expectation} by the following argument.
Let $X_t$ have the same distribution as $X$, and take the expectations of the equations in \Cref{expectation} over $X_t$ to yield
\[ \expect{\expect{f(X_{t+1}|X_t}}=-p\expect{f(X_t)} \quad\text{and}\quad \expect{\expect{g(X_{t+1}|X_t}}=p^2\expect{g(X_t)}. \]
By the law of iterated expectations, the two left-hand sides are $\expect{f(X_{t+1})}$ and $\expect{g(X_{t+1})}$, respectively.
As $\pi$ is stationary, $X_{t+1}$ is also the size of a set that is distributed according to $\pi$, so we have $\expect{f(X_{t+1})}=\expect{f(X_t)}=\expect{f(X)}$ and similarly for $g$.
It follows that $\expect{f(X)}=-p\expect{f(X)}$ and $\expect{g(X)}=p^2\expect{g(X)}$, so $\expect{f(X)}=\expect{g(X)}=0$.
\end{proof}

\begin{proof}[Proof of \Cref{cutoff}, Part (2)]
To make computations easier, we let $h(x)= \sqrt{n} f(x)$. Let $X$ be the size of a random subset of $P$ that is distributed according to $\pi$. We have
\[\expect{h(X)}=0,\]
so \Cref{variance} gives
\[\varn(h(X))= p.\]
Chebychev's inequality implies that
\begin{equation}\label{stationary}
\mathbb P\left(|h(X)|\leq\frac{p^{-c}}{2}\right) \geq 1- 4 p^{2c+1}. 
\end{equation}
Let $t=\frac{1}{2} \log_{1/p}n -c$. \Cref{expectation,variance} give
\[
\expect{h(X_t) \vert X_0=0 }=\sqrt{n}(-p)^tf(0)=\pm p^{-c}
\]
and
\[\varn(h(X_t)\vert X_0=0)= p-p^{2t}+ (1-p)(-p)^t = p \pm \frac{1-p}{\sqrt{n}}p^{-c} -\frac{p^{-2c}}{n}.\]
Suppose $0\leq c\leq \log_{1/p}n$. Another application of Chebychev's inequality and the fact that
\[ \varn(h(X_{t})\vert X_0=0)\leq p + \frac{1-p}{\sqrt{n}}p^{-c} \leq p + (1-p) = 1 \]
yield
\begin{align}
     \mathbb P\left(\vert h(X_t) \vert \leq \frac{p^{-c}}{2}\middle \vert X_0=0\right) \leq \mathbb P\left(\vert  h(X_t) - \expect{h(X_t)}\vert \geq \frac{ p^{-c}}{2} \middle 
\vert X_0=0\right)
& \label{transition} \leq 4 p^{2c}.
\end{align}
Combining \eqref{stationary} and \eqref{transition}, we get that
\[\dTV(Q^t(\emptyset,\cdot),\pi) \geq 1- 4 p^{2c+1}-4 p^{2c}. \qedhere \]
\end{proof}

\section{Future Directions}\label{sec:conclusion}

In \cref{cutoff}, we proved that the rowmotion Markov chains of Boolean lattices exhibit the cutoff phenomenon. It would be very interesting to obtain similar results for other toggle Markov chains. Some particularly interesting toggle Markov chains ${\bf T}(\KK,{\bf x})$ are as follows: 
\begin{itemize}
\item Let $P$ be the set of vertices of a graph $G$, let $\KK$ be the collection of independent sets of $G$, and let ${\bf x}$ be some special ordering of $P$. For example, if $G$ is a cycle graph, then ${\bf x}$ could be the ordering obtained by reading the vertices of $G$ clockwise. 
\item Let $P$ be an $n$-element set, and let ${\bf x}$ be an arbitrary ordering of the elements of $P$. For $0\leq k\leq n$, let $\KK=\{I\subseteq P:|I|\leq k\}$.  
\item Let $P$ be an $n$-element set, and let ${\bf x}$ be an arbitrary ordering of the elements of $P$. For $0\leq k\leq n$, let $\KK=\{I\subseteq P:|I|\geq k\}$.  
\end{itemize}
It would also be interesting to improve our estimates for the mixing times of rowmotion Markov chains for other families of semidistrim (or just distributive) lattices. 

In \Cref{thm:toggle_stationary,thm:hexx}, we computed the stationary distributions of rowmotion Markov chains of distributive lattices and the lattices $\hexx_{a,b}$. It would be quite interesting to find other special families of semidistrim lattices for which one can compute these stationary distributions.

Defant and Williams \cite{Semidistrim} found a close relationship between the rowmotion and \emph{pop-stack sorting} operators of a semidistrim lattice. In \cite{Ungarian}, the first two authors explored \emph{Ungarian Markov chains}, which are defined by introducing randomness into the definition of pop-stack sorting; one can view the Ungarian Markov chain of a semidistrim lattice $L$ as an absorbing analogue of the rowmotion Markov chain of $L$. 

\section*{Acknowledgments}
Colin Defant was supported by the National Science Foundation under Award No.\ 2201907 and by a Benjamin Peirce Fellowship at Harvard University. Evita Nestoridi was supported by the National Science Foundation  grant DMS-2052659. We thank the anonymous referee for helpful advice that greatly improved this article.

\end{document}